\documentclass[11pt]{amsart}

\usepackage{bm}
\usepackage{fullpage}
\usepackage{amssymb}
\usepackage{amsmath,amsfonts,amsthm}
\usepackage{hyperref}
\usepackage{color}
\usepackage{cite}

\newtheorem{alphtheorem}{Theorem}

\newtheorem{alphlemma}{Lemma}

\newtheorem{theorem}{Theorem}
\newtheorem{problem}{Problem}[theorem]
\newtheorem{question}{Question}[theorem]
\newtheorem{lemma}[theorem]{Lemma}
\newtheorem{fact}[theorem]{Fact}

\newtheorem{claim}[theorem]{Claim}
\newtheorem{proposition}[theorem]{Proposition}

\newtheorem{conjecture}{Conjecture}

\DeclareMathOperator\G{\mathcal{G}}

\DeclareMathOperator\F{\mathcal{F}}

\title{Borodin-Kostochka conjecture and Partitioning a graph into classes with no clique of specified size}

\author{Yaser Rowshan$^1$}
\keywords{ Vertex partitionable, Clique number, Degenerate graphs, Borodin-Kostochka  conjecture.}
\subjclass[2010]{05C69, 05C35, 05C15.}
\address{$^1$Y. Rowshan, 
	Department of Mathematics, Institute for Advanced Studies in Basic Sciences (IASBS), Zanjan 45137-66731, Iran}
\email{y.rowshan@iasbs.ac.ir,~y.rowshan.math@gmail.com}

\begin{document}
	\maketitle 
	\begin{abstract}
For a given graph $H$ and the  graphical properties  $P_1, P_2,\ldots,P_k$, a graph $H$	is said to be $(V_1, V_2,\ldots,V_k)$-partitionable if there exists a partition of $V(H)$ into $k$-sets $V_1, V_2\ldots,V_k$,  such that for each $i\in[k]$, the subgraph induced by  $V_i$ has the property $P_i$.  In $1979$, Bollob\'{a}s and Manvel  showed that for a graph $H$  with  maximum degree $\Delta(H)\geq 3$ and clique number $\omega(H)\leq \Delta(H)$, if $\Delta(H)= p+q$, then there exists a $(V_1,V_2)$-partition of $V(H)$, such that $\Delta(H[V_1])\leq p$, $\Delta(H[V_2])\leq q$, $H[V_1]$ is $(p-1)$-degenerate, and $H[V_2]$ is $(q-1)$-degenerate. 

Assume  that  $p_1\geq p_2\geq\cdots\geq p_k\geq 2$  are  $k$ positive integers and  $\sum_{i=1}^k p_i=\Delta(H)-1+k$. Assume that for each $i\in[k]$ the properties $P_i$  means that $\omega(H[V_i])\leq p_i-1$. Is $H$  a $(V_1,\ldots,V_k)$-partitionable graph?

  In 1977, Borodin and Kostochka  conjectured that any graph $H$ with maximum degree $\Delta(H)\geq 9$ and without $K_{\Delta(H)}$ as a subgraph, has chromatic number at most $\Delta(H)-1$. 	Reed proved that  the conjecture holds whenever $ \Delta(G) \geq 10^{14} $.	 
	When  $p_1=2$ and $\Delta(H)\geq 9$, the above question is the Borodin and Kostochka conjecture. Therefore, when  all $p_i$s are equal to $2$ and $\Delta(H)\leq 8$, the answer to the above question is negative. Let $H$ is a  graph with maximum degree $\Delta$, and clique number $\omega(H)$, where $\omega(H)\leq \Delta-1$.	In this  article, we intend to study this question when $k\geq 2$ and $\Delta\geq 13$. In particular  as an analogue of the Borodin-Kostochka conjecture, for the case that $\Delta\geq 13$ and $p_i\geq 2$ we prove that  the above question is true.
	\end{abstract}
	
	\section{Introduction}   
	All graphs considered in this article are undirected, simple, and finite. For a given  graph  $H=(V(H),E(H))$, the degree and neighbors of $v\in V(H)$ are denoted by $\deg_H{(v)}$ ($\deg{(v)}$) and $N_H(v)$($N(v)$), respectively. The  maximum degree of $H$ is denoted by  $\Delta(H)$ and the minimum degree of $H$ is denoted by $\delta(H)$. 	Suppose that $V'$ and $V''$ be two disjoint subsets of $V(H)$, we use $E[V',V'']$ to denote the set of edges between the bipartition $(V', V'' )$ of $H$. Suppose that $W$ is any subset of $V(H)$, the induced subgraph $H[W]$ is the graph whose vertex set is $W$ and whose edge set consists of all of the edges in $E(H)$ that have both endpoints in $W$.  The clique number $\omega(H)$ of a graph $H$ is the largest integer $\omega$, so that $H$ contains a complete subgraph of size $\omega$. 
	For a given graph $H$ and the  graphical properties  $P_1, P_2,\ldots,P_k$, a graph $H$	is said to be $(V_1, V_2,\ldots,V_k)$-partitionable if there exists a partition of $V(H)$ into $k$-sets $V_1, V_2\ldots,V_k$,  such that for each $i\in[k]$, the subgraph induced by  $V_i$ has the property $P_i$.  For more results on  $(V_1,\ldots, V_k)$-partitionable see e.g. \cite{simoes1976joins,bickle2021maximal,  bickle2021wiener, bickle2012structural, von2022point, lick1970k}.  	A graph $H$ is $k$-degenerate if every subgraph of $H$ contains a vertex of degree at most $k$. For the case that $k=2$, Bollob\'{a}s and Manvel  in  \cite{Bollob} have shown  the following result regarding  $(V_1,V_2)$-partition.
	\begin{alphlemma}\label{lem1}{\rm\cite{Bollob}}
		Suppose that $H$ is a graph with  maximum degree $\Delta(H)\geq 3$ and clique number $\omega(H)\leq \Delta(H)$.
		If $\Delta(H)= p+q$, then there exists a $(V_1,V_2)$-partition of $V(H)$, such  that $\Delta(H[V_1])\leq p$, $\Delta(H[V_2])\leq q$, $H[V_1]$ is $(p-1)$-degenerate, and $H[V_2]$ is $(q-1)$-degenerate.  
	\end{alphlemma}	
	As a generalization of Brooks' Theorem, Catlin showed that every graph $H$ with $\Delta(H)\geq 3$
	without $K_{\Delta(H)+1}$ as a subgraph, has a $\Delta(H)$-coloring such that one of  the color classes is 
	a maximum independent set~\cite{Catlin}. The author and Taherkhani,	in \cite{rowshan2022catlin}, have shown the following theorem.
	\begin{alphtheorem}\label{2th}\cite{rowshan2022catlin}
		Let $d_1,d_2\ldots,d_k$ be $k$  positive integers. Assume that $G_1,G_2\ldots,G_k$
		are  connected graphs with minimum degrees  $d_1,d_2\ldots,d_k$, respectively, and $H$ is a connected graph with maximum degree
		$\Delta(H)$ where $\Delta(H)=\sum_{i=1}^{k}d_k$. Assume that $G_1,G_2,\ldots,G_k$,
		and $H$ satisfy the following conditions:
		\begin{itemize}
			\item If $k=1$, then $H$ is not isomorphic to  $G_1$.
			\item If  $G_i$ is isomorphic to $K_{d_i+1}$ for each $1\leq i\leq k$, then $H$ is not isomorphic to $K_{\Delta(H)+1}$.
			\item If  $G_i$ is isomorphic to $K_{2}$ for each $1\leq i\leq k$, then $H$ is neither an odd cycle nor a complete graph.
		\end{itemize}
		Then, there is a partition of vertices of $H$ to $V_1,V_2\ldots,V_k$
		such that each $G_i\nsubseteq H[V_i]$ and moreover one of  $V_i$s  can be chosen in a way that $H[V_i]$  be a maximum induced  $G_i$-free subgraph in $H$(i.e $H$ contain  no copy of $G$). 
	\end{alphtheorem}
In this article, for a graph $H$ and $k$ positive integers $p_1\geq p_2\geq\cdots\geq p_k\geq 2$, instead to $(V_1,\ldots, V_k)$-partitionable,	we say $H$ is $(K_{p_1},\ldots,K_{p_k})$-partitionable if there exists  a partition of $V(H)$ into $k$-sets, $V_1, V_2\ldots,V_k$, such that for each $i\leq k $ we have  $\omega(H[V_i])\leq p_i-1$.	
		\begin{question}~\label{genBK}
		Suppose that $H$ is a graph with $\Delta(H)\geq 6$ and clique number $\omega(H)$, where $\omega(H)\leq \Delta(H)-1$. 
		Assume  that  $p_1\geq p_2\geq\cdots\geq p_k\geq 2$  are  $k$ positive integers 
		and  $\sum_{i=1}^k p_i=\Delta(H)-1+k$. Is $H$    a $(K_{p_1},\ldots,K_{p_k})$-partitionable graph?
	\end{question}
In $1977$, Borodin and Kostochka conjectured that Brooks’ bound can be further improved if $\Delta(H)\geq 9$.
\begin{conjecture}\label{co1}\cite{borodin1977upper}
For a graph $H$, if  $\Delta(H)\geq 9$, then $\chi(H)\leq \max\{\omega(H),\Delta(H)-1\}$.
\end{conjecture}
By Brooks'theorem, each graph $H$ with $\chi(H) > \Delta(H)\geq 9$  contains $K_{\Delta(H)+1}$. So the
Borodin-Kostochka conjecture is equivalent to the statement that each $H$ with $\chi(H)= \Delta(H)\geq 9$  contains $K_{\Delta(H)}$. In $1999$, Reed proved that  the Conjecture \ref{co1} holds whenever $ \Delta(G) \geq 10^{14} $ \cite{reed1999strengthening}. Note that  the above question for $k=1$ is clearly true.  Also,	note that  the above question in case that  $p_1=2$ (i.e. all $p_i$s are equal to $2$) and $\Delta(H)\geq 9$, is the Borodin and Kostochka conjecture. Therefore, when  all $p_i$s are equal to $2$ and $\Delta(H)\leq 8$, the answer to the above question is negative. Let $\G$ be a  class of graphs such that the Borodin-Kostochka conjecture is true for each $H\in \G$. Therefore for each $H\in\G$ we have $\chi(H) \leq \Delta-1$. Hence  by $\sum_{i=1}^k (p_i-1)=\Delta(H)-1$  it can be said that for each $H\in \G$,  $V(H)$ can be decomposed into $k$-sets $V_1, V_2\ldots,V_k$, such that for each $i\leq k $, we have  $\chi(H[V_i])\leq p_i-1$, that is   $\omega(H[V_i])\leq p_i-1$ for each $i\leq k $.  Therefore, if  the Borodin and Kostochka conjecture holds and in Question~\ref{genBK}  we have $p_1\geq 3$ and $\Delta(H)\geq 9$, then the answer of Question~\ref{genBK} is 
positive.

So assume that $H$ is a  counter-example to the Borodin-Kostochka conjecture of maximum degree $\Delta$, and clique number $\omega(H)$, where $\omega(H)\leq \Delta-1$. When $\chi(H)=\Delta(H)$, there are  series of interesting and useful results  that show $\omega(H)$
	must be close to $\Delta(H).$ As the first result, Borodin and Kostochka showed that if $\chi(H) =\Delta(H) \geq 7$, then $H$ contains $K_{\frac{\Delta(H)+1}{2}}$ \cite{BORODIN1977247}.  Mozhan proved that $\omega(H)\geq \Delta(H)-3$, when $\Delta(H)\geq 31$\cite{mozhan}. Finally, Cranston and Rabern strengthened the Mozhan's result   by weakening the condition to $\Delta(H)\geq 13$ \cite{cranston2015graphs}. 
	
	  In this  article, we shall show the following theorems are valid.
	\begin{theorem}~\label{thm1}
		Assume  that $\Delta$, $p_1\geq p_2\geq\cdots\geq p_k\geq 2$  are  $k+1$ positive integers where $p_1+p_2\geq 14$ and  $\sum_{i=1}^k p_i=\Delta-1+k$. If $H$ is a  countere-example for the Borodin-Kostochka conjecture with $\Delta(H)=\Delta\geq 13$, then $H$ is  a $(K_{p_1},\ldots,K_{p_k})$-partitionable.
	\end{theorem}
	\begin{theorem}~\label{t2}
		Suppose that $H$ is a graph with $\Delta(H)\geq 13$, $\chi(H)=\Delta(H)$, and clique number $\omega(H)\leq \Delta(H)-1 $. Also, suppose that  $p_1\geq p_2\geq\cdots\geq p_k\geq 2,$  are  $k$ positive integers where   $\sum_{i=1}^k p_i=\Delta(H)-1+k$. Let $V'$ is a subset of $V(H)$, such that $H[V']\cong K_{\omega(H)}$. If  $H'=H\setminus V'$ is  a $(K_{p_1},\ldots,K_{p_k})$-partitionable and moreover one of  $V_i$s  can be chosen in a way that $H'[V'_i]$  be a maximum induced   $K_{p_1}$-free subgraph in $H'$,  $H$ is  a $(K_{p_1},\ldots,K_{p_k})$-partitionable and moreover one of  $V_i$s  can be chosen in a way that $H[V_i]$  be a maximum induced   $K_{p_1}$-free subgraph in $H$.
	\end{theorem}
	As a direct consequence of Theorem~\ref{thm1}, and  Theorem~\ref{t2}, one can see that if $k=2$, $p_1=p_2=p\geq 7$, $ \Delta\geq 13$, and  $H$ be a graph with $\omega(H)\leq \Delta(H)-1$, then $\chi_{_{K_p}}(H)\leq \lceil{\Delta(H)-1\over p-1}\rceil=2< \lceil{\Delta(H)\over p-1}\rceil=3$, which is better than the result of Theorem  \ref{2th}. Another proof of  theorem \ref{thm1} can be seen in  \cite{yas}.
	\section{2-vertex partition of a graph  into two induced subgraphs not containing prescribed cliqes}	
	In this section, we prove the main results for the case that $k=2$.	  In particular we prove the following theorem.
	\begin{theorem}~\label{t1}
		Suppose that $H$ be a graph with $\Delta(H)\geq 13$ and clique number $\omega(H)$, where  $\omega(H)\leq \Delta(H)-1$. 
		Assume  that  $p$ and $q$ are  two positive integers where $p+q=\Delta(H)+1$. Then $H$ is a  $(K_{p_1},K_{p_2})$-partitionable.
	\end{theorem}
	
	With the assumptions of Theorem  \ref{t1}, if $\chi(H)\leq \Delta-1$, then it can be said that  there exists a $(V_1,V_2)$-decomposition of $V(H)$, so that $\chi(H[V_1])\leq p-1$ and $\chi(H[V_2])\leq q-1$. Hence, in this case the correctness of Theorem \ref{t1}, is obvious. Hence, we may suppose that $\chi(H)=\Delta$. That is assume that $H$ is a  counter-example to the Borodin-Kostochka conjecture. By the following results of  Cranston and  Rabern, we can suppose that  $\omega(H)\in \{ \Delta-3, \Delta-2, \Delta-1\}$. 
	
	\begin{theorem}\label{th2}\cite{cranston2015graphs}
		If $H$ is a  graph with $ \chi(H)\geq\Delta(H)\geq 13$, then $\omega(H)\geq\Delta(H)-3$.
	\end{theorem}
	To prove  Theorem \ref{t1}, we need  the following results.
	\begin{proposition}\label{l1}
		Let $H$ be a graph, where $\Delta(H)=\Delta$ and $\omega(H)=\omega$. Also, assume that $K_{\omega}$ is a clique of size $\omega$ in $H$. Hence, for each $v,v'\in V(H)\setminus V(K_{\omega})$, such that $vv'\in E(H)$, there exists two members $w,w'(w\neq w')$ of $V(K_{\omega})$, such that $vw,v'w'\notin E(H)$. 
	\end{proposition}
	\begin{proof}
		In otherwise, one can find a clique of size at least $\omega+1$ in $H$, which is a contradiction.	
	\end{proof}	
	The next result was proved by Reed in \cite{reed1999strengthening}.  
	\begin{lemma}\label{l0}\cite{reed1999strengthening}
		If $H$ is a minimal counter-example to the Borodin-Kostochka conjecture of maximum degree $\Delta$, then its $\Delta-1$ cliques are disjoint.
	\end{lemma}
	Now, we will use the idea used to prove Lemma \ref{l0}, and prove the following theorem.
	\begin{theorem}\label{th1}
		If $H$ is a minimal counter-example to Theorem \ref{t1} for the case that $\Delta(H)=\Delta$,  $\omega(H)=\Delta-1$ and $p\geq q\geq 5$, then its $\Delta-1$ cliques are disjoint.
	\end{theorem}
	\begin{proof}
	 To prove, we consider the following two claims.
		\begin{claim}\label{c11} No two $\Delta-1$ cliques  intersect in  $\Delta-2$ vertices.
		\end{claim} 
		\begin{proof}[Proof of Claim~\ref{c11}]
			Suppose that there are two  $\Delta-1$ cliques, say $K$ and $K'$, which intersect in a $\Delta-2$ clique $K''$ and assume that $x\in K\setminus K''$, $y\in K'\setminus K''$. We note that each of these two vertices is adjacent to
			every element of $K''$, that is every vertex of $K''$ has degree $\Delta-1$ in $K\cup K'$. Therefore, $|N(z)\cap (H\setminus (K\cup K'))|\leq 1$ for each $z\in  V(K'')$. Now, there is on edge  $xy\notin E(H)$,  otherwise this copy along with $K''$ forms a $\Delta$ clique, a contradiction to $\omega(H)=\Delta-1$. Therefore, no $\Delta-1$ clique  intersects in  $x$ and $y$. Let $H'= H\setminus \{x,y\}$. By the minimality of $H$,  there exaists a $(S_1,S_2)$-decomposetion of $V(H')$, where $H'[S_1]$ is $K_{p}$-free and $H'[S_2]$ is $K_{q}$-free. Also, as $|K''|=\Delta-2$, without loss of generality we can  suppose that  $|S_1\cap V(K'')|=|V_1|=p-1$  and $|S_2\cap V(K'')|=|V_2|=q-2$. Now, consider $(W_1,W_2)$-decomposetion of $V(H)$, where $W_1=S_1$ and $W_2=S_2\cup\{x,y\}$. If $H[W_2]$ is $K_{q}$-free,  which contradicts the assumption that $H$ is a minimal counter-example and the proof is complete. So, we may  suppose that there exists at least one copy of $K_q$ in $ H[W_2]$. Since  $xy\notin E(H)$, no $q$-clique  intersects in  $x$ and $y$ in $H[W_2]$. Therefore there exists a vertex $w$ of $S_2\setminus V_2$,  such that $V_2\subseteq N(w)$ and $w$ is adjacent to at least one vertex of $\{x,y\}$ in $H$. Without loss of generality, assume that $xw\in E(H)$, that is  $H[V_2\cup \{x,w\}]\cong K_{q}$. In this case, we claim that $V_1\subseteq N(w)$. Otherwise, if there exists a vertex $v$ of $V_1$, such that $vw\notin E(H)$, then we set $S'_1=(S_1\setminus \{v\}) \cup \{u\}$ and $S'_2=(S_2\setminus \{u\}) \cup \{v\}$, where $u\in V_2$. Now, since $vw\notin E(H)$,  $V_2\setminus \{v\}\subseteq N(w)$, $|N(z)\cap S_2|\leq 1$ for each $z\in  V_2$	and $q \geq 5$,  it is easy to check that $H[S'_2]$  is $K_{q}$-free. Also, since $|N(z)\cap (S_1\setminus V_1)|\leq 1$ for each $z\in  V_1$ and $|N(u)\cap S_2|=q+1$, we have $|N(u)\cap S_1|=|V_1\setminus \{v\}|=p-2$, that is  $H[S'_1]$  is $K_{p}$-free. Which contradicts the assumption that $H$ is a minimal counter-example. Hence, assume that $V_1\subseteq N(w)$, that is $H[V(K'')\cup\{x,w\}]=K_{\Delta}$, a contradiction again. Thus, for any vertex $w$ of $S_2$, we have $V_2\nsubseteq N(w)$. Therefore as $xy\notin E(H)$ and  $|N(z)\cap (H\setminus (K\cup K'))|\leq 1$ for each $z\in  V(K'')$, one can say that $H[W_2]$ is  $K_{q}$-free, it contradicts the assumption that $H$ is a minimal counter-example. Therfore the cliam is true.
		\end{proof}
		\begin{claim}\label{c12}  No two $\Delta-1$ cliques  intersect in  $\Delta-3$ vertices.
		\end{claim}
		\begin{proof}[Proof of Claim~\ref{c12}]
			Suppose that there are two  $\Delta-1$ cliques, say $K$ and $K'$, which intersect in a $\Delta-3$ clique $K''$. Also,  assume that $x,y\in K\setminus K''$ and $x',y'\in K'\setminus K''$. We note that each of these four vertices is adjacent to
			every element of $K''$.  Therfore $|N(z)\cap (H\setminus (K\cup K'))|=0$ for each $z\in V(K'')$. Now, there is no copy of $K_3$ in the induced graph  by $\{x,x',y,y'\}$. Otherwise this copy along with $K''$ forms a $\Delta$ clique, a contradiction to $\omega(H)= \Delta-1$. Thus, we can assume that $xx', yy'\notin E(H)$. Therefore,  no $\Delta-1$ clique  intersects in  $\{x,x'\}$ and $\{y,y'\}$. Let $H''=H \setminus (K\cup K')$, by the minimality of $H$, there exists a  $(S'_1,S'_2)$-decomposetion of $V(H'')$, sush that $H''[S'_1]$ is $K_{p}$-free and $H''[S'_2]$ is $K_{q}$-free. Also, since $|K''|=\Delta-3$, one cane assume that there exists a  $(V_1,V_2)$-decomposetion of $V(K'')$, such that  $H[V_1]\cong K_{p-2}$ and $H[V_2]\cong K_{q-2}$. Now, set $H'=H''\cup K''= H\setminus \{x,x',y,y'\}$ and    concider $(S_1,S_2)$-decomposetion of $V(H')$ such that  $S_i=S'_i\cup V_i$ for $i=1,2$. Since every vertex of $K''$ has degree $\Delta$ in $K\cup K'$ for rach $i=1,2$, one can check that  $H'[V_i]$ is a connected component in $H'[S_i]$. Now,  consider $W_1=S_1\cup\{x,x'\}$ and $W_2=S_2\cup\{y,y'\}$. Since  $xx', yy'\notin E(H)$, and no $\Delta-1$ clique  intersects in  $\{x,x'\}$ and $\{y,y'\}$, also as $\Delta\geq 9$, $p\geq 5,q\geq 5$,  $|N(z')\cap S_i|\leq 2$ for each $z'\in \{x,x',y,y'\}$ and $|N(z)\cap S'_i|=0$ for each $z\in V(K'')$ and each $i\in\{1,2\}$, it is easy to check that $H[W_1]$  is $K_{p}$-free and $H[W_2]$ is $K_{p}$-free, which contradicts the assumption that $H$ is a minimal counter-example. Hence the claim is true.
		\end{proof}
		
		By $\Delta(H)=\Delta$,  it can be said that,  no two $\Delta-1$ cliques  intersect in fewer than $\Delta-3$ vertices. Otherwise, there is a vertex of $H$ with  degree at least $\Delta+1$, which is impossible. Therefore, by Claims \ref{c11} and \ref{c12}  the proof is complete.
	\end{proof}
	Now, in the following results by using    Theorem \ref{th1}, we  prove that   Theorem \ref{t1} holds  when   $\omega(H)=\Delta-1$ and $p\geq q\geq 5$.
	\begin{fact}\label{f1}
		Theorem \ref{t1} holds for the case that  $\omega(H)=\Delta-1$, and $p\geq q\geq 5$.
	\end{fact}
	
	\begin{proof}[\bf Proof of  Fact \ref{f1}]
		Without loss of generality we may assume that $H$ is a minimal counter-example of Theorem \ref{t1} with maximum
		degree $\Delta$  and $\omega(H)=\Delta-1$. Let  $K$ is a $\Delta-1$ clique of $H$. Set $H'=H\setminus K$. Hence,  each vertex of $K$ has at most two adjacents in $H'$ and by Theorem \ref{th1}, any vertex of $H'$ has at most $\Delta-3$ neighbors in $K$. Since $H$ is  a minimal counter-example, there exists a $(W_1,W_2)$-decomposition of $V(H')$, such that $H'[W_1]$ is $K_p$-free and $H'[W_2]$ is $K_q$-free. Also, as  $|K|=\Delta-1$, there exists a $(V_1,V_2)$-decomposition of $V(K)$, such that $|V_1|=p-1$ and   $|V_2|=q-1$. So, it is clear that  $H[V_1]$ is $K_{p}$-free and $H[V_2]$ is $K_{q}$-free. Set $\F$ as follows.
		\[\F=\{(V',V'')~~|~~V'\cap V''=\emptyset,~V'\cup V''=V(K),~ |V'|=p-1, |V''|=q-1\}.\]
		For each $i\in [|\F|]$, define $e_i=|E(H[W_1,V_i])|+|E(H[W_2,V'_i])|$ where,  $(V_i, V'_i)$ is a member of $\F$.  	Without loss of generality we may suppose that $e_1=|E(H[W_1,V_1])|+|E(H[W_2,V_1'])|\leq |E(H[W_1,V_i])|+|E(H[W_2,V_i'])|$ for each $(V_i, V_i')\in \F$. Now, we have the following claims.
		\begin{claim}\label{c13}  If there exists a vertex $v$ of $V_1(V'_1)$, such that $|N(v)\cap W_1(W_2)|=2$, then for each $v'\in V'_1(V_1)$ we have $|N(v')\cap W_2(W_1)|=0$. Also, if there exists a vertex $v$ of $V_1(V'_1)$, such that $|N(v)\cap W_1(W_2)|= 1$, then for each $v'\in V'_1(V_1)$ we have $|N(v')\cap W_2(W_1)|\leq 1$.
		\end{claim} 
		\begin{proof}[Proof of Claim~\ref{c13}]	
			Without loss of generality assume that 	$|N(v)\cap W_1|=2$ for some $v\in V_1$ and on contrary, suppose that there exists a vertex $v'$ of $V_1'$, such that $|N(v')\cap W_2|\geq 1$. In this case, one can check that $(V_2,V'_2)=(V_1\setminus\{v\}\cup \{v'\}, V'_1\setminus\{v'\}\cup \{v\})\in \F$. Also as each vertex of $K$ has at most two adjacents in $H'$, then it is easy to say that $e_2=|E(H[W_1,V_2])|+|E(H[W_2,V_2'])|\leq e_1-2$, which is a contradiction to the  minimality of $e_1$. For other cases the proof is the same and the claim is true.
		\end{proof}		
		
		\begin{claim}\label{c14}   If there exist two vertices $w,w'$  of $W_1$,  such that:
			\[|N(w) \cap N(w')\cap V_1|\geq p-2, \]
		or	if there exist two vertices $w,w'$  of $W_2$ such that:
			\[|N(w) \cap N(w')\cap V_1'|\geq q-2, \]
			then the proof is complete. 
		\end{claim} 
		\begin{proof}[Proof of Claim~\ref{c14}]	
			Without loss of generality assume that  $|N(w_1) \cap N(w_2)\cap V_1|\geq p-2$ for some $w_1, w_2\in W_1$. Also, we may suppose that $M=\{v_1,\ldots, v_{p-2}\}\subseteq N(w_1) \cap N(w_1)\cap V_1$. 	Since $|M|=p-2\geq 3$,  by Claim \ref{c13}, we have $|N(v')\cap W_2|=0$ for each $v'\in V_1'$. Since each vertex of $K$ has at most two adjacents in $H'$ and $\{w_1,w_2\}\subseteq N(v)$ for each $v\in M$ and as  any vertex of $H'$ has at most $\Delta-3$ neighbors in $K$, there exist at least two vertices $v_1'',v_2''$ of $V(K)\setminus M$,  such that $w_iv''_i\notin E(H)$ for $i=1,2$. Now, we set $(V_2 ,V'_2)$ as follows.
			\[(V_2,V'_2)=(M\setminus \{v_1\}\cup \{v''_1,v''_2\}, V(K)\setminus V_2).\]
			One can say that $(V_2,V'_2)\in \F$. As $|M|\geq 3$,   by Claim \ref{c13}, we have $|N(v')\cap W_2|=0$ for at least $q-2$ members of   $V_2'$, which means that $H[ V'_2\cup W_2]$ is $K_q$-free. Also, as $v''_1,v''_2\in V_2$, $w_1v''_1,w_2v''_2 \notin E(H)$, $M\setminus\{v_1\}\subseteq N(w_i)$, each vertex of $V_2$ has at most two adjacents in $W_1$, and $p\geq 5$, it is easy to check that  $H[ V_2\cup W_1]$ is $K_p$-free.  Hence, the proof is complete.
		\end{proof} 
		Now, by the assumption that $H$ is a minimal counter-example of Theorem \ref{t1}, one can assume that there exists a copy of $K_p$ in $H[V_1\cup W_1]$ or there is a copy of $K_q$ in $H[V_2\cup W_2]$. 	Without loss of generality we may assume that $K'=K_p\subseteq H[V_1\cup W_1]$. By considering the vertices of $V(K')$ and by Claim \ref{c14}, one can assume that $|V(K')\cap V_1|=p-1$.  	Without loss of generality assume that  $w_1\in  V(K')\cap W_1$, that is  $V_1\subseteq N(w_1)$. 	Now, by Claim \ref{c13},  we have $|N(v')\cap W_2|\leq 1$ for each $v'\in V_2'$. Since each vertex of $K$ has at most two adjacents in $H'$, any vertex of $H'$ has at most $\Delta-3$ neighbors in $K$, and $w_1v\in E(H)$ for each $v\in V_1$,   there exist at lest two vertices $v_1',v_2'$ of $V'_1$, such that $w_1v'_i\notin E(H)$. Now, we set $(V_2, V'_2)$ as follows.
		\[(V_2,V'_2)=(  V_1\setminus \{v_1\})\cup \{v'_1\}, V(K)\setminus V_2).\]
		One can say that $(V_2,V'_2)\in \F$.	If $H[V_2\cup W_1]$ be $K_p$-free and $H[V'_2\cup W_2]$ be $K_q$-free, then the proof is complete. Also, if there exists a copy of $K_p$ in $H[V_2\cup W_1]$, then as $V_2\setminus \{v'_1\}\subseteq  N(w_1)$, one can check that there exist two vertices $w_1,w_2$ of $W_1$, such that $|N(w_1) \cap N(w_2)\cap V_2|= p-2$ and so the proof is complete by Claim \ref{c14}.  So, assume that $H[V_2\cup W_1]$ is $K_p$-free and $H[V'_2\cup W_2]$ has a copy $K''$ of $K_q$. By Claim \ref{c13} and as $|N(v')\cap W_2|\leq 1$ for each $v'\in V'_2$, one can assume that there exists a member $w'_1$ of $W_2$, so that $V'_2\subseteq N(w_2)$ and $V(K'')= V'_2\cup \{w_1'\}$.  As  any vertex of $H'$ has at most $\Delta-3$ neighbors in $K$, there exists at least one vertex $v_2$ of $V_2\setminus \{v'_1\}$, such that $w'_1v_2\notin E(H)$. Now, we set $(V_3,V'_3)$ as follows.
		\[(V_3,V'_3)=(  V_2\setminus \{v_2\})\cup \{v_1\}, V(K)\setminus V_2).\]		
		One can say that $(V_3,V'_3)\in \F$. Since $V_3\setminus \{v'_1\}\subseteq  N(w_1)$,  $|N(v)\cap W_1|\leq 1$ for each $v\in V_1$, and $p\geq 5$,  one can say that   $H[V_3\cup W_1]$ is $K_p$-free. Also, as $V'_3\setminus \{v_2\}\subseteq  N(w'_1)$,  $|N(v')\cap W_2|\leq 1$ for each $v'\in V_1'$, and $q\geq 5$,   one can check that   $H[V'_3\cup W_2]$ is $K_q$-free, which contradicts the assumption that $H$ is a minimal counter-example. Hence the assumption does not holds,  and the proof of the fact is complete.		
	\end{proof}
	By Fact \ref{f1}, Theorem \ref{t1} holds for the case that  $\omega(H)=\Delta-1$. So, one can suppose that  $\omega(H)\leq \Delta-2$. Hence, by using  Fact \ref{f1}, we can prove the following  result.
	\begin{fact}\label{fa1}
		Theorem \ref{t1} holds for the case that  $H$ is non-$\Delta$-regular graph, $\omega(H)\leq \Delta-2$, and $p\geq q\geq 5$.
	\end{fact}
\begin{proof}[\bf Proof of  Fact \ref{fa1}]
		As 	  $H$ is a  non-$\Delta$-regular graph, there exists at least one member $v$ of $V(H)$, such that $\deg(v)\leq \Delta-1$. Consider the complete graph $K=K_{\Delta-1}$  and let $y$ be a vertex of $K$. Let $H'=(H\cup K)+vy$. It is clear that  $H'$ is a graph with $\Delta(H')=\Delta$ and $\omega(H')=\Delta-1$. Therefore, by Fact \ref{f1}, there exists a $(V'_1,V'_2)$-decomposition of $V(H')$, such that $\omega(H'[V'_1])\leq p-1$ and $\omega(H'[V'_2])\leq q-1$. Hence, as $H\subseteq H'$, one can say that there exists a $(V_1,V_2)$-decomposition of $V(H)$, such that $\omega(H[V_1])\leq p-1$ and $\omega(H[V_2])\leq q-1$, where $V_1=V(H)\cap V'_1$ and $V_2=V(H)\cap V'_2$.
		
	\end{proof}
	
	Therfore, by Fact \ref{f1} and Fact \ref{fa1}, suppose that  $H$ is a  $\Delta$-regular graph with $\chi(H)=\Delta\geq 13$, and $\omega(H)\in \{ \Delta-2, \Delta-3\}$.
	
	\begin{fact}\label{f2}
		Theorem \ref{t1} holds for the case that $H$ is a  $\Delta$-regular graph with $\chi(H)=\Delta$, $\omega(H)=\Delta-2$,  and $p\geq q\geq 6$.
	\end{fact}
	
	\begin{proof}[\bf Proof of  Fact \ref{f2}]
		Without loss of generality we may suppose that $H$ is a minimal counter-example of Theorem \ref{t1} where $H$ is a  $\Delta$-regular, $\chi(H)=\Delta$, and $\omega(H)=\Delta-2$. Suppose that $K$ is a $\Delta-2$ clique of $H$. Set $H'=H\setminus K$. Therefore, any member of $K$ has  three adjacents in $H'$ and by  $\omega(H)=\Delta-2$,  any member of $H'$ has at most $\Delta-3$ neighbors in $K$. Since $H$ is  a minimal counter-example,  one can suppose that there is a $(W_1,W_2)$-decomposition of $V(H')$, such that $H'[W_1]$ is $K_p$-free and $H'[W_2]$ is $K_q$-free.  Since $|K|=\Delta-2$, we can  decompose   $V(K)$ into $(V', V'')$, where $|V'|=p-1$ and $|V''|=q-2$, or $|V'|=p-2$ and $|V''|=q-1$. It is clear that  in any case, $H[V']$ is $ K_{p}$-free and $H[V'']$ is $ K_{q}$-free. Set $\F$ as follows.
		\[\F=\{(V',V'')~~|~~V'\cap V''=\emptyset,~V'\cup V''=V(K),~ p-2\leq |V'|\leq p-1, q-2\leq |V''|\leq q-1\}.\]
		For each $i\in [|\F|]$, assume that $e_i=|E(H[W_1,V_i])|+|E(H[W_2,V_i'])|$, where $(V_i,V'_i)$ is a member of $\F$. 	Without loss of generality we may suppose that $e_1=|E(H[W_1,V_1])|+|E(H[W_2,V'_1])|\leq |E(H[W_1,V_i])|+|E(H[W_2,V_i'])|$ for each $(V_i,V'_i)\in \F$. Now, we have the following claims.
		\begin{claim}\label{c15}  If there exists a vertex $v$ of $V_1(V'_1)$,  such that $|N(v)\cap W_1(W_2)|\geq 2$, then for each $v'\in V'_1(V_1)$, we have $|N(v')\cap W_2(W_1)|\leq 1$. Also,  if there exists a vertex $v$ of $V_1(V'_1)$,  such that $|N(v)\cap W_1(W_2)|=3$, then for each $v'\in V'_1(V_1)$, we have $|N(v')\cap W_2(W_1)|=0$.
		\end{claim} 
		\begin{proof}[Proof of Claim~\ref{c15}]	
			Without loss of generality assume that 	$|N(v)\cap W_1|=2$ for some $v\in V_1$. Also, on contrary, suppose that there exists a vertex $v'$ of $V_1'$,  such that $|N(v')\cap W_1|\geq 2$. Hence, one can check that $(V_2,V_2')=(V_1\setminus\{v\}\cup \{v'\}, V'_1\setminus\{v'\}\cup \{v\})\in \F$. Also, as any member of $K$ has  three adjacents in $H'$, then it is easy to say that $e_2=|E(H[W_1,V_2])|+|E(H[W_2,V_2'])|\leq e_1-2$, which is a contradiction to the minimality of $e_1$ and the proof is complete. For other cases the proof is the same and the claim is true.
		\end{proof}		
		\begin{claim}\label{c16}  If $|V_1|=p-1(~or~|V_2|=q-1)$, then for each vertex $v$ of $V_1(V_2)$, we have $|N(v)\cap W_1(W_2)|\leq 1$.
		\end{claim} 
		\begin{proof}[Proof of Claim~\ref{c16}]	
			Without loss of generality assume that  $|V_1|=p-1$ and	$|N(v)\cap W_1|\geq 2$ for some $v\in V_1$. Now, set $(V_2, V_2')=(V_1\setminus\{v\}, V'_1\cup \{v\})\in \F$. Hence, as any member of $V(K)$ has three neighbors in $V(H')$,  it is easy to say that $e_2=|E(H[W_1,V_2])|+|E(H[W_2,V_2'])|\leq e_1-1$, which is a contradiction to the  minimality of $e_1$.
		\end{proof}	
		So, as $p-2\leq |V_1|\leq p-1$ and $ q-2\leq |V_2|\leq q-1$,	without loss of generality we may assume that  $|V_1|=p-1$ and $ |V_2|=q-2$. Now, we have a claim as follows. 	
		\begin{claim}\label{c17}     $H[V_1\cup W_1]$ is $K_p$-free.
		\end{claim} 
		\begin{proof}[Proof of Claim~\ref{c17}]	
			On contrary, assume that $K'$ is a copy of $K_p$ in  $H[V_1\cup W_1]$. As $|V_1|=p-1$,  by Claim \ref{c16},  there should be a vertex $w_1$ of $W_1$,  such that $V_1\subseteq N(w_1)$ and  $|N(v)\cap W_1|= 1$ for each $v\in V_1$. Also, by Claim \ref{c15},  one can say that  $|N(v')\cap W_2|\leq  2$ for each $v'\in V_2$. As  any vertex of $W_1$ has at most $\Delta-3$ neighbors in $K$, there exists at least one vertex $v'_1$ of $V_2$, such that $w_1v'_1\notin E(H)$. Now, we set $(V'_1,V'_2)$ as follows.
			\[(V'_1,V'_2)=( V_1\setminus \{v_1\})\cup \{v'_1\}, V(K)\setminus V'_1).\]
			It is clear that $(V'_1,V'_2)\in \F$. Also, as $|V'_1|=p-1\geq 5$ and $w_1v'_1\notin E(H)$, by Claim \ref{c16},  it can be said that  $H[V'_1\cup W_1]$ is $K_p$-free. If  $H[V'_2\cup W_2]$ be $K_q$-free, then the proof is complete. So, we may assume that there exists a copy $K''$ of $K_q$  in $H[V'_2\cup W_2]$. For each $v'_i\in V'_2$, as $|N(v'_i)\cap W_2|\leq 2$ and $|V'_2|=q-2$, it is easy to say that there exist two members $w'_1,w'_2$ of $W_2$, such that $w'_1w'_2\in E(H)$ and $V_2'\subseteq N(w'_i)$ for $i=1,2$, that is $V(K'')=V'_2\cup \{w'_1,w'_2\}$. Therefore, we have  $|N(v')\cap W_2|= 2$ for each $v'\in V'_2$. As  any vertex of $W_2$ has at most $\Delta-3$ neighbors in $K$ and $V_2'\subseteq N(w'_i)$, then by Proposition \ref{l1}, there exist at least two members $y_1,y_2$ of $V'_1$, such that $w'_iy_i\notin E(H)$. 	Without loss of generality assume that $y_1\neq v'_1\in V'_1$, where $y_1w'_1\notin E(H)$.  Now, we set $(V''_1,V''_2)$ as follows.
			\[(V''_1,V''_2)=( V'_1\setminus \{y_1\})\cup \{v_1\}, V(K)\setminus V''_1).\]
			It is clear that $(V''_1,V''_2)\in \F$, $|V''_1|=p-1\geq 5$, $v'_1\in V''_1$ and $w_1v'_1\notin E(H)$. So, by Claim \ref{c16}, it can be said that  $H[V'_1\cup W_1]$ is $K_p$-free. Also, as $|N(v')\cap W_2|= 2$ for each $v'\in V''_2\setminus \{y_1\}$,  $w'_1w'_2\in E(H)$, $w'_1y_1\notin E(H)$, $V_2''\setminus \{y_1\}\subseteq N(w'_i)$ for $i=1,2$, and $q\geq 6$, one can say that $H[V''_2\cup W_2]$ is $K_q$-free and the proof is complete.
		\end{proof} 
		Now as $|V_1|=p-1$, by Claim \ref{c17} we have $H[V_1\cup W_1]$ is $K_p$-free. Therfore, by the assumption that $H$ is a minimal counter-example of Theorem \ref{t1}, we can assume that there exists a copy $K'$ of $K_q$ in $H[V_2\cup W_2]$. By considering the vertices of $V(K')$ and as any member of $K'$ has  three adjacents in $H'$,  one can assume that $q-3\leq |V(K')\cap V_2|\leq q-2$. First assume that $|V(K')\cap V_2|= q-3$, that is there exist three vertices $w'_i$, $i=1,2,3$ of $W_2$,  such that $w'_iw'_j\in E(H)$ and $|V_2\cap  N(w'_1)\cap  N(w'_2)\cap  N(w'_3)|\geq |V_2|-1= q-3$.	Without loss of generality assume that $V''_2=V_2\setminus \{v''_1\}\subseteq V_2\cap  N(w'_1)\cap  N(w'_2)\cap  N(w'_3)$.  By Claim \ref{c15}, one can say that  $|N(v)\cap W_1|=0$ for each $v\in V_1$. Also, as  $|V_2\cap  N(w'_1)\cap  N(w'_2)\cap  N(w'_3)|\geq |V_2|-1= q-3$ and $|V_2|= q-2$, then by Proposition \ref{l1}, it can be said that there exist  two members $\{v_1,v_2\}$ of $V_1$, such that $v_iw'_i\notin E(H)$. Now,  we define $(V'_1,V'_2)$ as follows.
		\[(V'_1,V'_2)=( V(K)\setminus V'_2, V_2\setminus\{v''_1,v''_2 \} \cup \{v_1,v_2\} ).\]
		It is clear that $(V'_1,V'_2)\in \F$, $|V'_1|=p-1\geq 5$, and $|N(v)\cap W_1|=0$ for at least $p-3$ vertices of $V'_1$. So,  it can be said that  $H[V'_1\cup W_1]$ is $K_p$-free. Also, as $q\geq 6$, and for each $v'\in V'_2\setminus \{v_1,v_2\}$, $|N(v')\cap W_2|= 3$ and $w'_iv'\in E(H)$ for each $i=1,2,3$, also as $v_iw'_i\notin E(H)$, one can say that $H[V'_2\cup W_2]$ is $K_q$-free and the proof is complete. So, we may suppose that  $|V(K')\cap V_2|= q-2$, that is there exist two vertices $w'_i$, $i=1,2$ of $W_2$, such that $w'_1w'_2\in E(H)$ and $V_2\subseteq  N(w'_i)$ for $i=1,2$. By Claim \ref{c15}, for each $v\in V_1$, we have $|N(v)\cap W_1|\leq 1$. Also, by Proposition \ref{l1} there exist two members $\{v_1,v_2\}$ of $V_1$, such that $v_iw'_i\notin E(H)$ for $i=1,2$.  In this case, we set $(V'_1,V'_2)$ as follows.
		\[(V'_1,V'_2)=( V(K)\setminus V_2, V_2\setminus\{v'_1,v'_2\} \cup \{v_1,v_2\} ).\] 
		It is clear that $(V'_1,V'_2)\in \F$. Also, as $|N(v')\cap W_2|\leq 3$, $w'_iv'\in E(H)$ for each $v'\in V'_2\setminus \{v_1,v_2\}$, and $v_iw'_i\notin E(H)$, one can say that $H[V'_2\cup W_2]$ is $K_q$-free. Now, assume that $H[V'_1\cup W_1]$ has a copy of $K_p$. Therefore, as $p, q\geq 6$ and by Claim \ref{c15}, we have $|N(v)\cap W_1|= 1$ for each $v\in V'_1\setminus \{v'_1,v'_2\}$. One can assume that there exists a vertex $w_1$ of $W_1$,  such that $V'_1\subseteq N(w_1)$. Hence, we should have $|N(v)\cap W_1|= 1$ for each $v\in V'_1$ and $|N(v')\cap W_2|= 2$ for each $v'\in V'_2\setminus\{v_1,v_2\}$. As $H[V_1\cup W_1]$ is $K_p$-free and  $V'_1\subseteq N(w_1)$, one can assume that $w_1v_i\notin E(H)$ for at least one $i\in\{1,2\}$. 	Without loss of generality  assume that $w_1v_1\notin E(H)$. Now, we set $(V''_1,V''_2)$ as follows.
		\[(V''_1,V''_2)=( V_1\setminus\{v_2\}\cup\{v'_1\},  V(K) \setminus V''_1).\] 
		It is easy to say that $(V''_1,V''_2)\in \F$, $|V''_1|=p-1\geq 5$, $v_1w_1\notin E(H)$, $V''_1\setminus \{v_1\}\subseteq N(w_1)$, and $|N(v)\cap W_1|= 1$ for each $v\in V''_1$. So,  it can be said that  $H[V''_1\cup W_1]$ is $K_p$-free. Also, as $|N(v')\cap W_2|= 2$ for each $v'\in V''_2\setminus \{v_2\}$, $ V''_2\setminus \{v_2\}\subseteq N(W'_i)$ for $i=1,2$, and  $w'_2v'_2\notin E(H)$, one can say that $H[V'_2\cup W_2]$ is $K_q$-free. Hence, the proof is complete.	
	\end{proof}	

	\begin{fact}\label{f3}
		Theorem \ref{t1} holds for the case that $H$ is a  $\Delta$-regular graph with $\chi(H)=\Delta$, $\omega(H)=\Delta-3$, and $p\geq q\geq 7$.
	\end{fact}
	
	\begin{proof}[\bf Proof of the Fact \ref{f3}]
		Without loss of generality we may assume that $H$ is a minimal counter-example of Theorem \ref{t1} where $H$ is   $\Delta$-regular, $\chi(H)=\Delta$, and $\omega(H)=\Delta-3$. Assume that $K$ be a $\Delta-3$ clique in $H$. Set $H'=H\setminus K$. Hence,  each vertex of $K$ has  four adjacents in $H'$ and by  $\omega(H)=\Delta-3$,  any vertex of $H'$ has at most $\Delta-4$ neighbors in $K$. Since $H$ is  a minimal counter-example, there exists a $(W_1,W_2)$-decomposition of $V(H')$, such that $H'[W_1]$ is $K_p$-free and $H'[W_2]$ is $K_q$-free.   Set $\F$ as follows.
		\[\F=\{(V',V'')~~|~~V'\cap V''=\emptyset,~V'\cup V''=V(K),~ p-3\leq |V'|\leq p-1, q-3\leq |V''|\leq q-1\}.\]
		
		For each $(V',V'')\in \F$, it can be said that  $H[V']$ is $ K_{p}$-free and $H[V'']$ is $ K_{q}$-free. 		For each $i\in [|\F|]$, define $e=|E(H[W_1,V'])|+|E(H[W_2,V''])|$, where  $(V', V'')$ is a member of $\F$.	Without loss of generality we may suppose that $e_1=|E(H[W_1,V_1])|+|E(H[W_2,V_1'])|\leq |E(H[W_1,V_i])|+|E(H[W_2,V_i'])|$ for each $(V_i, V_i')\in \F$. Now, we have the following claims.
		\begin{claim}\label{c18}  If there exists a vertex $v$ of $V_1(V'_1)$, such that $|N(v)\cap W_1(W_2)|\geq 3$, then for each $v'\in V'_1(V_1)$, we have $|N(v')\cap W_2(W_1)|\leq 1$. If $|N(v)\cap W_1(W_2)|\geq 2$, then for each $v'\in V'_1(V_1)$, we have $|N(v')\cap W_2(W_1)|\leq 2$. Also, if there exists a vertex $v$ of $V_1(V'_1)$, such that $|N(v)\cap W_1(W_2)|=4$, then for each $v'\in V'_1(V_1)$, we have $|N(v')\cap W_2(W_1)|=0$.
		\end{claim} 
		\begin{proof}[Proof of Claim~\ref{c18}]	
			Without loss of generality one can assume that 	$|N(v)\cap W_1|=3$ for some $v\in V_1$. Also, on contrary, suppose that there exists a vertex $v'$ of $V_1'$,  such that $|N(v')\cap W_2|\geq 2$. In this case, one can check that $(V_2, V_2')=(V_1\setminus\{v\}\cup \{v'\}, V'_1\setminus\{v'\}\cup \{v\})\in \F$. Also, as  each vertex of $K$ has  four adjacents in $H'$, then it can be said that $e_2=|E(H[W_1,V_2])|+|E(H[W_2,V_2'])|\leq e_1-2$, which is a contradiction to the minimality of $e_1$. For othar cases the proof is the same and the proof of the claim is complete.
		\end{proof}		
		\begin{claim}\label{c19}  If $|V_1|\geq p-2( ~or~|V_2|\geq q-2)$, then for each vertex $v$ of $V_1(V_2)$, we have $|N(v)\cap W_1(W_2)|\leq 2$.
		\end{claim} 
		\begin{proof}[Proof of Claim~\ref{c19}]	
			Without loss of generality assume that $|V_1|\geq p-2$, $v\in V_1$, and	$|N(v)\cap W_1|\geq 3$.  Now, set $(V_2, V_2')=(V_1\setminus\{v\}, V'_1\cup \{v\})\in \F$. As,  each vertex of $K$ has  four adjacents in $H'$, it is easy to say that $|N(v)\cap W_2|\leq 1$, which means that	$e_2=|E(H[W_1,V_2])|+|E(H[W_2,V_2'])|\leq e_1-2$, which contradicts the minimality of $e_1$ and the proof is complete.
		\end{proof}	
		Now, considr $|V_1|$ and $|V_2|$.  	Without loss of generality   assume that  $|V_1|\geq |V_2|$. Hence, we have two cases  as follows. 	
		
		{\bf Case 1}: $|V_1|=p-2$ and $ |V_2|=q-2$. In this case, by Claim \ref{c19}, for each $v\in V_1(V_2)$ we have $|N(v)\cap W_1(W_2)|\leq 2$. 	Now, by the assumption that $H$ is a minimal counter-example of Theorem \ref{t1}, we can assume that there exists a copy of $K_p$ in $H[V_1\cup W_1]$ or there exists a copy of $K_q$ in $H[V_2\cup W_2]$. 	Without loss of generality we may assume that there exists a copy $K'$ of $K_p$  in $H[V_1\cup W_1]$. Therefore, by using Claims \ref{c18} and \ref{c19} and the fact that $|V_1|=p-2$, there should be two members $w_1$ and $w_2$ of $W_1$, such that $V_1\subseteq N(w_i)$ and  $w_1w_2\in E(H)$.  Therefore, by Proposition \ref{l1},  there exist two members $v'_1,v'_2$ of $V_2$,  such that $w_iv'_i\notin E(H)$.  Now, we set $(V'_1,V'_2)$ as follows.
		\[(V'_1,V'_2)=(  V_1\setminus \{v_1,v_2\})\cup \{v'_1,v'_2\}, V(K)\setminus V'_1).\]
		It is clear that $(V'_1,V'_2)\in \F$ and $|V'_1|=p-2\geq 5$. As $w_iv'_i\notin E(H)$,  by Claim \ref{c19}, it can be said that  $H[V'_1\cup W_1]$ is $K_p$-free. If  $H[V'_2\cup W_2]$ be $K_q$-free, then the proof is complete. So, we may assume that there exists a copy $K''$ of $K_q$   in $H[V'_2\cup W_2]$. As $|V'_2|=q-2$ and $|N(v'_i)\cap W_2|\leq 2$ for each $v'_i\in V'_2$, it is easy to say that there exist two members $w'_1,w'_2$ of $W_2$, such that $w'_1w'_2\in E(H)$ and $V_2'\subseteq N(w'_i)$ for $i=1,2$, that is $V(K'')=V'_2\cup \{w'_1,w'_2\}$. Since   $w'_1w'_2\in E(H)$  and $V_2'\subseteq N(w'_i)$ for $i=1,2$, by Proposition \ref{l1} there exist at least two members $y_1,y_2$ of $V'_1$, such that $w'_iy_i\notin E(H)$. 	Without loss of generality assume that $y_1\neq v'_1\in V'_1$, where $y_1w'_1\notin E(H)$.  Now, we set $(V''_1,V''_2)$ as follows.
		\[(V''_1,V''_2)=( V'_1\setminus \{y_1\})\cup \{v_1\}, V(K)\setminus V''_1).\]
		It is clear that $(V''_1,V''_2)\in \F$. Also, as $|V''_1|=p-2\geq 5$,  $w_1v'_1\notin E(H)$, and  $|N(v')\cap W_1|= 2$ for each $v'_i\in V''_1\setminus \{v'_1\}$ and $V_1''\setminus \{v'_1\}\subseteq N(w_i)$ for $i=1,2$,  it can be said that  $H[V''_1\cup W_1]$ is $K_p$-free. Since $|N(v')\cap W_2|= 2$ for each $v'\in V''_2$,  $w'_1w'_2\in E(H)$, and $V_2''\setminus \{y_1\}\subseteq N(w'_i)$ for $i=1,2$, one can say that $H[V''_2\cup W_2]$ is $K_q$-free and the proof is complete.
		
		{\bf Case 2}: $|V_1|=p-1$ and $ |V_2|=q-3$. In this case, by Claim  \ref{c19}, one can say that  $|N(v)\cap W_1|\leq 2$ for each $v\in V_1$. If there exists a member $v$ of $V_1$,  such that $|N(v)\cap W_1|= 2$, then  $|N(v)\cap W_2|= 2$. Otherwise, we come to a contradiction  by the minimality of $e_1$ by considring $(V'_1,V'_2)=( V_1\setminus \{v\}, V(K)\setminus V'_1)$. So, assume that  $|N(v)\cap W_1|= 2$ and set $(V'_1,V'_2)=( V_1\setminus \{v\}, V(K)\setminus V'_1)$. One can say that $(V'_1,V'_2)\in \F$ and $e=e_1=|E(H[W_1,V_1])|+|E(H[W_2,V_2])|= e'=|E(H[W_1,V'_1])|+|E(H[W_2,V_2'])|$, therefore  the proof is complete by Case 1. Hence, we may suppose that $|N(v)\cap W_1|\leq 1$ for each $v\in V_1$. Now,  by assumption that $H$ is a minimal counter-example of Theorem \ref{t1}, we can assume that there exists a copy of $K_p$ in $H[V_1\cup W_1]$ or there exists a copy of $K_q$ in $H[V_2\cup W_2]$.  First, we may assume that there exists a copy $K'$ of $K_p$  in $H[V_1\cup W_1]$. Therefore, as $|V_1|=p-1$ and $|N(v)\cap W_1|\leq 1$ for each $v\in V_1$, there exiasts a member $w_1$ of $W_1$,  such that $V_1\subseteq N(w_1)$.  Therefore, by Proposation \ref{l1} there exists a member $v'_1$ of $V_2$, such that $w_1v'_1\notin E(H)$.  Now, we set $(V'_1,V'_2)$ as follows.
		\[(V'_1,V'_2)=(  V_1\setminus \{v_1\})\cup \{v'_1\}, V(K)\setminus V'_1).\]
		
		It is clear that $(V'_1,V'_2)\in \F$ and $|V'_1|=p-1\geq 6$. As  $w_1v'_1\notin E(H)$ and   $|N(v)\cap W_1|\leq 1$ for each $v\in V_1$, it can be said that  $H[V'_1\cup W_1]$ is $K_p$-free. If  $H[V'_2\cup W_2]$ is $K_q$-free, then the proof is complete. So, we may assume that there exists a copy $K''$ of $K_q$   in $H[V'_2\cup W_2]$. As $|N(v_i)\cap W_2|\leq 3$ and $|V'_2|=q-3$, it is easy to say that there exist three members $w'_1,w'_2,w'_3$ of $W_2$,  such that $w'_iw'_j\in E(H)$ and $V_2'\subseteq N(w'_i)$ for $i=1,2,3$, that is $V(K'')=V'_2\cup \{w'_1,w'_2,w'_3\}$. Now, one can say that there exists at least one member $y_1\neq v'_1 $ of $V'_1$,  such that $w'_iy_1\notin E(H)$ for one $i\in\{1,2,3\}$.  Assume that $y_1w'_1\notin E(H)$.  Now, we set $(V''_1,V''_2)$ as follows.
		\[(V''_1,V''_2)=( V'_1\setminus \{y_1\})\cup \{v_1\}, V(K)\setminus V''_1).\]
		It is clear that $(V''_1,V''_2)\in \F$, $|V''_1|=p-1\geq 6$, and $w_1v'_1\notin E(H)$. So,  it can be said that  $H[V''_1\cup W_1]$ is $K_p$-free. Also, as $|N(v')\cap W_2|= 3$ for each $v'\in V'_2$,  $w'_1w'_2\in E(H)$, and $V_2''\setminus \{y_1\}\subseteq N(w'_i)$ for $i=1,2,3$ and  $y_1w'_1\notin E(H)$, one can say that $H[V''_2\cup W_2]$ is $K_q$-free and the proof is complete.
		
		Now, we may assume that $H[V_1\cup W_1]$ is $K_p$-free, and there exists a copy of $K_q$ in $H[V_2\cup W_2]$. 	Without loss of generality we may assume that $K'=K_q\subseteq H[V_2\cup W_2]$. By considering the vertices of $V(K')$ and as each vertex of $K$ has  four adjacents in $H'$, one can assume that $q-4\leq |V(K')\cap V_2|\leq q-3$. Let $|V(K')\cap V_2|= q-4$, that is there exist four vertices $w'_i$, $i=1,2,3,4$ of $W_2$,  such that $w'_iw'_j\in E(H)$ and $|V_2\cap N(w'_1)\cap  N(w'_2)\cap  N(w'_3)|\geq q-4$. By Claim \ref{c18}, for each $v\in V_1$, we have $|N(v)\cap W_1|= 0$. Assume that $\{v'_1,v_2',v'_3\}\subseteq V_2\cap N(w'_1)\cap  N(w'_2)\cap  N(w'_3)$. Also, by Proposition \ref{l1} there exist three members $\{v_1,v_2,v_3\}$ of $V(K)$,  such that $v_iw'_i\notin E(H)$. Now, we set $(V'_1,V'_2)$ as follows.
		\[(V'_1,V'_2)=( V(K)\setminus V'_2 , V_2\setminus\{v'_1,v_2',v'_3\} \cup \{v_1,v_2,v_3\} ).\]
		It is clear that $(V''_1,V''_2)\in \F$,  $|V'_1|=p-1\geq 6$, and $|N(v)\cap W_1|=0$ for at least five vertices of $V'_1$. Hence,  it can be said that  $H[V'_1\cup W_1]$ is $K_p$-free. Also, as $|N(v')\cap W_2|= 4$,  $w'_iv'\in E(H)$ for each $v'\in V'_2\setminus \{v_1,v_2,v_3\}$, and $v_iw'_i\notin E(H)$, one can say that $H[V'_2\cup W_2]$ is $K_q$-free and the proof is complete. So, we may suppose that  $|V(K')\cap V_2= V_2|= q-3$, that is there exist three vertices $w'_i$, $i=1,2,3$ of $W_2$,  such that $w'_iw'_j\in E(H)$ and $V_2\subseteq  N(w'_i)$. By Claim \ref{c18}, for each $v\in V_1$, we have $|N(v)\cap W_1|\leq 1$  and by Proposition \ref{l1}, there exist three members $\{v_1,v_2,v_3\}$ of $V(K)$,  such that $v_iw'_i\notin E(H)$.  Now, we set $(V'_1,V'_2)$ as follows.
		\[(V'_1,V'_2)=( V(K)\setminus V'_2, V_2\setminus\{v'_1,v'_2,v'_3\} \cup \{v_1,v_2,v_3\} ).\]
		
		It is clear that $(V''_1,V''_2)\in \F$. As $q\geq 7$, $|N(v')\cap W_2|\leq 4$,  $w'_iv'\in E(H)$ for each $v'\in V'_2\setminus \{v_1,v_2,v_3\}$, and $v_iw'_i\notin E(H)$, one can say that $H[V'_2\cup W_2]$ is $K_q$-free. Now, assume that $H[V'_1\cup W_1]$ has a copy of $K_p$. Therefore, as $p\geq 7$ and for each $v\in V_1$ we have $|N(v)\cap W_1|\leq 1$, one can assume that there exists a vertex $w_1$ of $W_1$,  such that $V'_1\subseteq N(w_1)$. Hence, we should have $|N(v)\cap W_1|= 1$ for each $v\in V'_1$ and $|N(v')\cap W_2|\leq 3$ for each $v'\in V'_2$. As $H[V_1\cup W_1]$ is $K_p$-free and  $V'_1\subseteq N(w_1)$, one can assume that $w_1v_i\notin E(H)$ for at least one $i\in\{1,2,3\}$. 	Without loss of generality we may assume that $w_1v_1\notin E(H)$. Now, we set $(V''_1,V''_2)$ as follows.
		\[(V''_1,V''_2)=( V_1\setminus\{v_2,v_3\}\cup\{v'_1,v'_2\},  V(K) \setminus V''_1).\] 
		It is clear that $(V''_1,V''_2)\in \F$ and $|V'_1|=p-1\geq 6$. As $v_1w_1\notin E(H)$, $V''_1\setminus \{v_1\}\subseteq N(w_1)$, and $|N(v)\cap W_1|= 1$ for each $v\in V''_1$, and  $|V''_1|=p-1\geq 6$,  it can be said that  $H[V''_1\cup W_1]$ is $K_p$-free. Also, as $|N(v')\cap W_2|\leq 3$ for each $v'\in V''_2$, and  $w'_iv'_i\notin E(H)$ for $i=2,3$, one can say that $H[V'_2\cup W_2]$ is $K_q$-free. Therfore  the proof is complete.	
	\end{proof}
	By Fact \ref{f1}, Fact \ref{fa1} , Fact \ref{f2}, and Fact \ref{f3},  we prove that Theorem \ref{t1} is true for some cases.	To prove  another cases we need  the following results of Landon Rabern.
	\begin{theorem}\label{th5}\cite{rabern2011hitting}
		If $H$ is a  graph with $\omega(H)\geq \frac{3(\Delta+1)}{4}$, then $H$ has an independent set $I$ such that $\omega(H\setminus I)\leq \omega(H)-1 $.
	\end{theorem}
	\begin{theorem}\label{th6}\cite{king}
	If $H$ is a connected graph with $\omega(H)> \frac{2(\Delta+1)}{3}$, then $H$ has an independent set $I$ such that $\omega(H\setminus I)\leq \omega(H)-1 $, unless it is the strong product of an odd cycle with length at	least $5$ and the complete graph $K_{\frac{\omega(H)}{2}}$.
\end{theorem}
	In the following, we prove Theorem \ref{t1}.
	
	\begin{proof} [ \bf Proof of  Theorem \ref{t1}]  	Without loss of generality we may assume that $p\geq q$. If $\chi(H)\leq \Delta-1$ then the proof is trivial. So, assume that $\chi(H)= \Delta$. Therefore, by Theorem \ref{th2} we should have $\Delta-3\leq \omega(H)\leq \Delta-1$. When $\omega(H) = \Delta-1$, by Fact \ref{f1}, the proof is complete for each $p,q\geq 5$. When $\omega(H) = \Delta-2$, by Fact \ref{f2}, the proof is complete for each $p,q\geq 6$,  and for $q= 2$ the proof is trivial. Also, when $\omega(H) = \Delta-3$, by Fact \ref{f3}, the proof is complete for each $p,q\geq 7$,  and for $q\leq 3$ the proof is trivial. Therefore,  we need to show that the theorem holds  for the following  cases:
		\begin{itemize}
			\item   $\omega(H) = \Delta-1$ and $q\in \{2,3,4\}$,
			\item   $\omega(H) = \Delta-2$ and $q\in \{3,4,5\}$,
			\item   $\omega(H) = \Delta-3$ and $q\in \{4,5,6\}$.
		\end{itemize}
	
{\bf Case 1:} Assume that  $\omega(H) = \Delta-1$ and $q\in \{2,3,4\}$. In this case we have $(p,q)\in \{(\Delta-1,2), (\Delta-2,3), (\Delta-3,4)\}$.  First assume that $(p,q)=(\Delta-1,2)$. By Theorem \ref{th5}, one can say that  $H$ has a    independent set $I_1$, such that $\omega(H\setminus I_1)\leq  \Delta-2$. By setting $(V_1, V_2)=(V(H)\setminus I_1, I_1)$, the proof is complete. Now assume that $(p,q)=(\Delta-2,3)$.  Let $I_1$ be  a maximum independent set of $H$. Therefore, we have  $\Delta(H\setminus I_1) \leq \Delta-1$. If  $\omega(H\setminus I_1)\leq  \Delta-3$, the proof is complete. So, assume that  $\omega(H\setminus I_1)=\Delta-2$. Therefore as $\Delta\geq 13$, by Theorem \ref{th5},  $H\setminus I_1$ has a   independent set $I_2$, such that $\omega(H\setminus I_1)\leq  \Delta-3$. By setting $(V_1, V_2)=(V(H)\setminus (I_1\cup I_2), I_1\cup I_2)$, the proof is complete. For the case that $(p,q)=(\Delta-3,4)$ the proof is the same.

{\bf Case 2:} Assume that  $\omega(H) = \Delta-2$ and $q\in \{3,4,5\}$. In this case we have $(p,q)\in \{(\Delta-2,3), (\Delta-3,4), (\Delta-4,5)\}$.  For  $(p,q)=(\Delta-2,3)$ the proof is complete by Case 1. Now assume that $(p,q)=(\Delta-3,4)$, for other case the proof is same.  Let $I_1$ be  a maximum independent set of $H$ and $I_2$ be  a maximum independent set of $H\setminus I_1$ . Therefore, we have  $\Delta(H\setminus (I_1\cup I_2)) \leq \Delta-2$. If  $\omega(H\setminus (I_1\cup I_2))\leq  \Delta-4$, the proof is complete. So, assume that  $\omega(H\setminus (I_1\cup I_2))=\Delta-3$. Therefore as $\Delta\geq 13$ and $\Delta(H\setminus (I_1\cup I_2))\leq \Delta-2$, by Theorem \ref{th5} and Theorem \ref{th6}, one can say that  $H\setminus (I_1\cup I_2)$ has a    independent set $I_3$, such that $\omega(H\setminus (I_1\cup I_2))\leq  \Delta-4$. By setting $(V_1, V_2)=(V(H)\setminus (I_1\cup I_2\cup I_3), I_1\cup I_2\cup I_3)$, the proof is complete.
	
{\bf Case 3:} Assume that $\omega(H) = \Delta-3$ and $q\in \{4,5,6\}$. Therefore, consider the case that  $\omega(H) = \Delta-3$ and $(p,q)\in \{(\Delta-3,4),(\Delta-4,5), (\Delta-5,6)\}$. First assume that $(p,q)=(\Delta-3,4)$.  Let $I_1$ be  a maximum independent set of $H$. Therefore, we have  $\Delta(H\setminus I_1) \leq \Delta-1$. If  $\omega(H\setminus I_1)\leq  \Delta-4$, the proof is complete. So, assume that  $\omega(H\setminus I_1)=\Delta-3$. Hence as $\Delta \geq 13$, by Theorem \ref{th5} one can say that  $H\setminus I_1$ has a   independent set $I_2$, such that $\omega(H\setminus I_1)\leq  \Delta-4$. By setting $(V_1, V_2)=(V(H)\setminus (I_1\cup I_2), I_1\cup I_2)$, the proof is complete. Now, assume that $(p,q)=(\Delta-4,5)$. For $(p,q)=(\Delta-5,6)$ the proof is the same. Let $I_1$ is  a maximum independent set of $H$. Therefore,   $\Delta(H\setminus I_1) \leq \Delta-1$. Assume that  $\omega(H\setminus I_1)\leq  \Delta-4$, and set  $I_2$ is a maximum independent set  of $H\setminus I_1$. Therefore, $\Delta(H\setminus (I_1\cup I_2)) \leq  \Delta-2$. If  $\omega(H\setminus (I_1\cup I_2))\leq \Delta-5$,  the proof is complete. So, assume that $\omega(H\setminus (I_1\cup I_2))=\Delta-4$. Now,  as $\Delta(H\setminus (I_1\cup I_2))\leq \Delta-2$ and $\Delta\geq 13$, so by Theorem \ref{th6}, one can say that  $H\setminus( I_1\cup I_2)$ has a  independent set $I_3$, such that $\omega(H\setminus (I_1\cup I_2\cup I_3))\leq \Delta-5$. By setting $(V_1, V_2)=(V(H)\setminus (I_1\cup I_2\cup I_3), I_1\cup I_2\cup I_3)$ the proof is complete. Therfore, we may assume that  $\omega(H\setminus I_1)= \Delta-3$.  So by Theorem \ref{th5},  $H\setminus I_1$ has a  independent set $I_2$, such that $\omega(H\setminus (I_1\cup I_2))\leq \Delta-4$. Therefore, $\Delta(H\setminus (I_1\cup I_2)) \leq \Delta-2$. If  $\omega(H\setminus (I_1\cup I_2))\leq \Delta-5$,  the proof is complete. So, assume that $\omega(H\setminus (I_1\cup I_2))=\Delta-4$. So  as $\Delta(H\setminus (I_1\cup I_2))\leq \Delta-2$ and $\Delta\geq 13$,  by Theorem \ref{th6}  one can checked that   $H\setminus (I_1\cup I_2)$ has a  independent set $I_3$, such that $\omega(H\setminus (I_1\cup I_2\cup I_3))\leq \Delta-5$. By setting $(V_1, V_2)=((V(H)\setminus (I_1\cup I_2\cup I_3)), I_1\cup I_2\cup I_3)$ the proof is complete. 

Hence, by Case 1, Case 2, and Case 3 the theorem holds.
	\end{proof}	
	
	\section{Proof of  Theorem \ref{thm1} and Theorem \ref{t2}}	
	In this section, we prove the main results.
	\begin{proof} [ \bf Proof of  Theorem \ref{thm1}]
		First note that  we only need to prove the statement for $k=2$. By Theorem \ref{t1} one can said that  the statement holds  for $k=2$. To prove Theorem~\ref{thm1},
		set $p=\sum_{i=1}^{k-1}p_i-(k-2)$ and $q=p_k$. Hence $p+q=\Delta(H)+1$. Since the statement is true for $k=2$, we have 
		a partition of $V(H)$ into  $(V_1, V_2)$ such that $H[V_1]$ is $K_p$-free and  $H[ V_2] $ is $K_q$-free. Also, assume that 
		$V_2$ is a maximall $K_q$-free subset of $H$, such that for which we have $H[V_1]$ is $K_p$-free. Therefore, each vertex in $V_1$ has at least $q-1$ neighbours  in $V_2$.
		Thus, the maximum degree of $H[V_1]$ is at most $\Delta(H)-(q-1)=p$.  
		
		If the maximum degree of $H[V_1]$ is less than $p$, let $v\in V_1$ be a vertex with maximum degree in $H[V_1]$ such that its degree is equal to $m<p$. We add $\{x_1,\ldots, x_{p-m}\}$ new vertices to $H[V_1]$ and join all of them to $v$, forming a new graph $G$. The graph $G$ has maximum degree $p$.  Also it can be say that $G$ is $K_p$-free graph. Therefore by  induction $G$ is a $(V_1,\ldots, V_{k-1})$-partitionable graph. 
	Hence, there exists a partition of $V(G)$   into  $V'_1,\cdots, V'_{k-1}$ such that for each $1\leq i\leq k-1$, $G[V'_i]$ is $K_{p_i}$-free. 
		Therefore it can be say that $V'_1\cap V_{1},\ldots, V'_{k-1}\cap V_1, V_2$ is the desired partition of $V(H)$, that is $H$ is a $(V_1,\ldots, V_{k})$-partitionable graph. Therefore, we may assume that $H[V_{1}]$  is a graph with maximum degree $p$.  	We also have $\omega( H[V_{1}])\leq p-1$, $p_1+p_2\geq 14$, and $\sum_{i=1}^{k-1}p_i=\Delta(H[V_1])-1+(k-1)$. We iterate this procedure until we  obtain the desired partition.

		 As $k\geq 3$ and $p_1+p_2\geq 14$, we can assume that
		$H[V_1]$ has maximum degree $p\geq 13$. Now, we have $\sum_{i=1}^{k-1}p_i=p-1+(k-1)$. We iterate this procedure until to  obtain the required partition. Hence, the proof is complete. 
	\end{proof}		
	
	\begin{proof} [ \bf Proof of  Theorem \ref{t2}] 	First note that  we only need to prove the statement for $k=2$. Assume that the statement  holds for $k=2$, i.e. if $\Delta(H)-3\leq \omega(H)\leq \Delta(H)-1$, $p,q\geq 2$,  and $p+q=\Delta(H)+1\geq 14$, then there exists a partition of $V(H)$ into  $(V_1, V_2)$ such that $H[V_1]$ is a maximum $K_p$-free subgraph and  $H[ V_2] $ is $K_q$-free. To prove Theorem~\ref{t2}, set $q=\sum_{i=2}^{k}p_i-(k-2)$ and $p=p_1$. Since the statement is true for $k=2$, we have a partition of $V(H)$ into  $(V_1, V_2)$ such that $H[V_1]$ is a maximum $K_p$-free subgraph and  $H[ V_2] $ is $K_q$-free. As $H[V_1]$ is a maximum $K_p$-free subgraph, each vertex in $V_2$ has at least $p-1$ neighbours  in $V_1$. Thus, the maximum degree of $H[V_2]$ is at most $\Delta(H)-(p-1)=q$. We can assume that $H[V_2]$ has maximum degree $q\geq 13$. Now, we have $\sum_{i=2}^{k}p_i=p-1+(k-1)$. We iterate this procedure until to  obtain the required partition. 	
		
		Now, we prove the statement   for $k=2$. 	Without loss of generality we may assume that $p\geq q$. On contrary, we may suppose that $H$ is a minimal counter-example of Theorem \ref{t2} with maximum degree $\Delta$ and  $\omega(H)=t$, where $t\in \{\Delta-3, \Delta-2, \Delta-1\}$. Let $K$ is a $t$ clique in $H$. Set $H'=H\setminus K$.  Since each vertex of $K$ has at most theree adjacents in $H'$ and   for each $v\in M$ and as  any vertex of $H'$ has at most $t-1$ neighbors in $K$, then one can assume that $\Delta(H)-t+1\leq \Delta(H')\leq \Delta(H)$. 	If the maximum degree of $H'$ is less than $\Delta(H)$, let $v\in V(H')$ be a vertex with maximum degree in $H'$ such that its degree is equal to $m<\Delta(H)$. We add $X'=\{x_1,\ldots, x_{m}\}$ new vertices to $H'$ and join all of them to $v$, forming a new graph $G$, where $m\leq t-1$ and $m+\Delta(H')=\Delta(H)$. The graph $G$ has maximum degree $\Delta(H)$.  Also it can be say that $|V(G)|\leq |V(H)|-1$ and $\omega(G)\leq \Delta(G)-1$.
		
		 Since $H$ is  a minimal counter-example, there exists a $(W_1,W_2)$-decomposition of $V(G)$, such that $G[W_1]$ is a maximum $K_p$-free subgraph and $G[W_2]$ is $K_q$-free. Assume that $(Y_1, Y_2)=(W_1\cap V(H'),W_2\cap V(H')$. As $G[X']$ is independent, it easy to say that  $(Y_1, Y_2)$  is a decomposition of $V(H')$, such that in $H'$, $H'[Y_1]$ is a maximum $K_p$-free subgraph   and $H'[Y_2]$ is $K_q$-free.
		 
		    Now, consider $(S_1,S_2)$-decomposition of $V(H)$, such that $S_1= Y_1\cup V_1$  and $S_2= Y_2\cup V_2$ where $|V_1|=p-1, |V_2|=t-p+1$, and $V_1\cup V_2=V(K)$. By Theorem \ref{t1}, one can assume that $H[S_1]$ is  $K_p$-free  and $H[S_2]$ is $K_q$-free. We claim that $H[S_1]$ is a maximum $K_p$-free subgraph of $H$. Otherwise, assume that there exists a subset $S$  of $V(H)$ with $|S|\geq |S_1|+1$, such that $H[S]$ is a  $K_p$-free subgraph and $H[V(H)\setminus S]$ is $K_q$-free. As $|K|=t$ and $H[S]$ is a  $K_p$-free, it can be said that $|S\cap V(K)|\leq p-1$, and $t-p+1\leq |(V(H)\setminus S)\cap V(K)|\leq q-1$. Therefore, one can say that  there exists a $(Y'_1,Y'_2)$-decomposition of $V(H')$, such that $H'[Y'_1]$ is  $K_p$-free,  $H'[Y'_2]$ is $K_q$-free, and $|Y'_1|\geq |Y_1|+1$, consequently  there exists a $(W'_1,W'_2)$-decomposition of $V(G)$, such that $G[W'_1]$ is  $K_p$-free,  $G[W'_2]$ is $K_q$-free, and $|W'_1|\geq |W_1|+1$, were $(W'_1,W'_2)=(Y'_1\cup X',Y'_2)$,    which contradicts the fact that  $G[W_1]$ is a maximum $K_p$-free subgraph of $G$.  Therefore, $H[S_1]$ is a maximum $K_p$-free subgraph of $H$ and that the proof is complete.
	\end{proof}		
	\subsection{ Some research problems related to the contents of this paper.}

	In this section, we propose some research problems related to the contents of this paper. The first problem concerns  Theorem \ref{t1}, Theorem \ref{t2} and Theorem \ref{2th}, as we address below.
	\begin{problem}
		Let $p$ and $q$ are two positive integers,   where   $ p, q\geq 2$ and $p+q=\Delta+1$. Suppose that $H$ is a graph with  $\Delta(H)=\Delta\leq 12$ and $\omega(H)\leq \Delta-1$. If $\Delta+1= p+q$, then there exists a $(V_1,V_2)$-decomposition of $V(H)$, so that $H[V_1]$ is a   $K_p$-free subgraph and $H[V_2]$ is a $K_q$-free subgraph of $H$. 
		
	\end{problem}
	\begin{problem}
		Let $p$ and $q$ are two positive integers,   where   $ p, q\geq 3$ and $p+q=\Delta+1$. Suppose that $H$ is a graph with  $\Delta(H)=\Delta\geq 5$ and $\omega(H)\leq \Delta-1$. There exists a $(V_1,V_2)$-decomposition of $V(H)$, so that $H[V_1]$ is a maximum $K_p$-free subgraph and $H[V_2]$ is a $K_q$-free subgraph of $H$. 
		
	\end{problem}
	
	\begin{problem}
		Let $p$ and $q$ are two positive integers,   where   $ p, q\geq 3$ and $p+q=\Delta+1$. Suppose that $H$ is a graph with  $\Delta(H)=\Delta\geq 5$ and $\omega(H)\leq \Delta-1$. There exists a $(V_1,V_2)$-decomposition of $V(H)$, so that $H[V_1]$ is  $(p-1)$-regular-free subgraph and $H[V_2]$ is $(q-1)$-regular-free subgraph of $H$. 
	\end{problem}
		\section{Declarations}
	{\bf Conflict of Interest:} On behalf of all authors, the corresponding author states	that there is no conflict of interest.
	
	{\bf Data Availability Statement:}No data were generated or used in the preparation of this paper.

	\bibliographystyle{plain}
	\bibliography{G-free} 
\end{document}